\documentclass[a4paper,11pt]{article}
\usepackage[all]{xy}
\title{Normal triangulations in o-minimal structures}

\author{El\'ias Baro \thanks{Partially supported by
GEOR MTM2005-02568}\\
Departamento de Matem\'aticas\\ Universidad Aut\'onoma de Madrid\\ 28049 Madrid, Spain}

\usepackage{amsmath}
\usepackage{amscd}
\usepackage{amsfonts}
\usepackage{amssymb}
\usepackage{amsthm}

\newtheorem{deff}{Definition}[section]
\newtheorem{teo}[deff]{Theorem}
\newtheorem{teoICVC}{TISSP-Triangulation Theorem}

\newtheorem{teoNT}{NT-Triangulation Theorem}

\newtheorem{lema}[deff]{Lemma}
\newtheorem{cor}[deff]{Corollary}
\newtheorem{prop}[deff]{Proposition}

\theoremstyle{definition}

\newtheorem{nota}[deff]{Notation}
\newtheorem{obs}[deff]{Observation}
\newtheorem{ejem}[deff]{Example}

\newcommand{\IC}{TIP}
\newcommand{\VC}{SSP}

\begin{document}
\maketitle
\begin{abstract}
We work over an o-minimal expansion of a real closed field $R$. Given a closed simplicial complex $K$ and some definable subsets $S_1,\ldots, S_k$ of its realization $|K|$ in $R$ we prove that there exists a triangulation $(K',\phi')$ of $|K|$ compatible with $S_1,\ldots,S_k$ such that $K'$ is a subdivision of $K$ and $\phi'$ is definably homotopic to $id_{|K|}$.
\end{abstract}

          %
          %
          %
          %
          %
\section{Introduction}
We work over an o-minimal expansion of an ordered field $R$. In the study of definable sets, the triangulation theorem is many times applied to find a triangulation of a definable set  compatible with some definable subsets,  $S_1,\ldots,S_k$ say, where the set is already the realization in $R$  of a simplicial complex $K$. In this particular case, it is natural to ask if we can find such a triangulation using $K$, that is, if there exists a subdivision $K'$ of $K$ --so that  the realizations $|K'|$ and $|K|$ coincide-- and a definable homeomorphism $\phi':|K'|\rightarrow |K|$ such that $(K',\phi')$ is a triangulation of $|K|$ compatible with $S_1,\ldots,S_k$.

In this paper we find such subdivision $K'$ and definable homeomorphism $\phi'$.  In fact, we find such a subdivision and $\phi'$ --we call it normal triangulation (see Definition \ref{def:normal})-- with $\phi'$ definably homotopic to $id_{|K|}$ (see Corollary \ref{cor:principal}). To do this we use a refinement of the proof of the Triangulation theorem in \cite{MR1633348} and a result of  o-minimal homology theory. In order to prove the existence of normal triangulations we also consider triangulations with some other special properties. For instance, we say that a triangulation $(K,\phi)$ of a closed and bounded definable set $X$ has the triangulation independence property (\IC) if, roughly, the image by $\phi'$ of the vertices of $K$ span a copy of $K$ (see Definition \ref{def:IC}). We will prove that for every closed and bounded definable set $X$ and some definable subsets $S_1,\ldots,S_k$ of $X$ there exists a triangulation of $X$ compatible with $S_1,\ldots,S_k$ having TIP (see Theorem \ref{tissptoerema}).

As a first application of the existence of normal triangulations, we will prove in \cite{miopreprint} that the obvious map between the  semialgebraic homotopy set and the o-minimal homotopy set of two semialgebraic sets is a bijection. This will allow us to transfer the whole semialgebraic homotopy theory developed by H.Delfs and M.Knebusch in \cite{MR819737} to the o-minimal setting.

In section \ref{section:2} we introduce the main concepts of our study, as normal triangulations and \IC,  we give some examples and we state the main results. In section \ref{section:3} we prove the existence  and basic properties of triangulations satisfying \IC. In section \ref{section:4} and \ref{section:5} we respectively show  the existence and properties of normal triangulations. For basic results on o-minimality we refer to \cite{MR1633348}.

The results of this paper are part of the author's  Thesis.
              %
             %
            %
           %
          %
\section{Definitions and examples}\label{section:2}

Let $\mathcal{R}$ be an o-minimal expansion of a field $R$. By ''definable'' we mean definable in $\mathcal{R}$. All functions are assumed to be continuous.  In this section we are going to introduce triangulations with some special properties. Firstly, we fix some notation.
\begin{nota}Recall the definition of a simplicial complex in 8.1.5 in \cite{MR1633348}, where the simplices are the smallest convex subsets whose closure contains certain affinely independent points in $R^m$. That is, we will consider the simplicial complexes as the realizations of abstracts complexes whose simplices are open. Given a definable set $S$ and some definable subsets $S_1,\ldots,S_k$ of $S$ we say that $(K,\phi)$
is a triangulation of $S$ compatible with $S_1,\ldots,S_k$, denoted by $(K,\phi)\in \Delta(S;S_1,\ldots,S_k)$, if $\phi:|K|\rightarrow S$ is a definable homeomorphism and $S_i$ is the union of the images of simplices of $K$ by $\phi$. Given $C=\phi(\sigma)$,
$\sigma=(v_{0},\ldots,v_{n}) \in K$, we say that $v$ is a vertex
of $C$ if $C=\phi(v_{i})$ for some $i=0,\ldots,n$. If $S$ is a definable set then
$\partial S= \overline{S}-S$ denotes its frontier and 
$Bd(S)=\overline{S}-int(S)$ its boundary.
\end{nota}

\begin{deff}\label{def:normal}Let $K$ be a closed simplicial complex in $R^{m}$ and $S_{1},\ldots,S_{k}$ definable subsets of $|K|$. A triangulation $(K',\phi')\in \Delta(|K|;S_{1},\ldots,S_{k})$ is a \textbf{normal triangulation of the complex $K$ compatible with $S_{1},\ldots,S_{k}$}, denoted by NT-triangulation or by $(K',\phi')\in \Delta^{NT}(|K|;S_{1},\ldots,S_{k})$, if
\begin{enumerate}
\item[i)] $(K',\phi')\in \Delta(|K|;S_{1},\ldots,S_{k},\sigma)_{\sigma\in K}$, 

\item[ii)] $(K',id)\in \Delta(|K|;\sigma)_{\sigma\in K}$, and

\item[iii)] for every $\tau \in K'$ and $\sigma \in K$ such that $\tau \subset \sigma$ we have that $\phi'(\tau)\subset \sigma$.

\end{enumerate}
If $k=0$ we say that $(K',\phi')$ is a
NT-triangulation of the complex $K$.
\end{deff}
\begin{obs}Following the notation of the definition of the NT-triangulation we observe that given empty subsets  $S_{1},\ldots, S_{k}$, that is, $k=0$, NT-triangulations are easy to obtain because any subdivision $(K',id)$ of $K$ is a NT-triangulation of $K$. We are mostly interested in NT-triangulations when the
 subsets $S_{1},\ldots,S_{k}$ are not empty.
\end{obs}

In order to prove the existence of the NT-triangulations we need to introduce triangulations with two properties.

\begin{deff}\label{def:IC}Let $(K,\phi)\in \Delta(S)$, where $S$ is a closed and bounded definable set in $R^{m}$. We say that
$(K,\phi)$ satisfies the \textbf{triangulation independence property (\IC)} if
\begin{enumerate}
\item[i)]for every $n$-simplex $\tau=(v_{0},\ldots,v_{n})\in K$ we have that \linebreak $\phi(v_{0}),\ldots,\phi(v_{n})\in R^{m}$ are affinely independent, that is, they span an $n$-simplex $\tau^{\phi}:=(\phi(v_{0}),\ldots,\phi(v_{n}))$ in $R^{m}$, and

\item[ii)]if $\tau_{1}$ and $\tau_{2}$ are different simplices of 
$K$ then $\tau_{1}^{\phi}$ and $\tau_{2}^{\phi}$ are disjoint.
\end{enumerate}
\end{deff}

\begin{deff}Let $S$ be a closed and bounded definable set and $S_{1},\ldots,S_{k}$ definable subsets of $S$. A triangulation $(K,\phi)\in \Delta(S;S_{1},\ldots,S_{k})$  satisfies \textbf{the small simplices property
 with respect to} $S_{1},\ldots,S_{k}$, denoted by \VC$(S_{1},\ldots,S_{k})$, if for every $\tau=(v_{0},\ldots,v_{n})\in K$ with
$\phi(v_{0}),\ldots,\phi(v_{n})\in \overline{S_{j}}$ we have that
$\phi(\tau)\subset \overline{S_{j}}$.
\end{deff}
\begin{ejem}All the examples are in dimension $2$.

$1)$ The following is  an example of a triangulation
$(K,\phi)$ of a closed and bounded definable set $S$ such that it does not satisfy \IC\ because $i)$ fails.
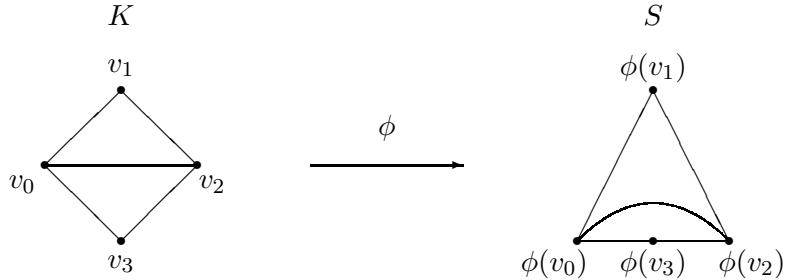
\begin{center}
 \setlength{\unitlength}{1cm}
\begin{picture}(10,5)
\put(0.5,2){\line(1,0){2}}

\put(0.5,2){\line(1,1){1}}

\put(0.5,2){\line(1,-1){1}}

\put(1.5,3){\line(1,-1){1}}

\put(1.5,1){\line(1,1){1}}

\put(1.5,4){\makebox(0,0){$K$}}

\put(0.5,2){\circle*{0.1}} \put(1.5,1){\circle*{0.1}}
\put(1.5,3){\circle*{0.1}} \put(2.5,2){\circle*{0.1}}

\put(0.2,1.75){\makebox(0,0){$v_{0}$}}
\put(2.7,1.75){\makebox(0,0){$v_{2}$}}
\put(1.5,3.3){\makebox(0,0){$v_{1}$}}
\put(1.5,0.7){\makebox(0,0){$v_{3}$}}
\put(7.5,1){\line(1,2){1}}

\put(7.5,1){\line(1,0){2}}

\put(8.5,3){\line(1,-2){1}}

\qbezier(7.5,1)(8.5,2)(9.5,1)

\put(8.5,4){\makebox(0,0){$S$}}

\put(7.5,1){\circle*{0.1}} \put(8.5,3){\circle*{0.1}}
\put(9.5,1){\circle*{0.1}} \put(8.5,1){\circle*{0.1}}

\put(7.2,0.7){\makebox(0,0){$\phi(v_{0})$}}
\put(8.5,0.7){\makebox(0,0){$\phi(v_{3})$}}
\put(8.5,3.3){\makebox(0,0){$\phi(v_{1})$}}
\put(9.8,0.7){\makebox(0,0){$\phi(v_{2})$}}

\put(5,2.5){\makebox(0,0){$\phi$}}

\put(4,2){\vector(1,0){2}}
\end{picture}
\end{center}
Observe that if we denote by $S_{1}=\phi((v_0,v_2))$ then $(K,\phi)$ has \VC($S_{1}$).

$2)$ The following is  an example of a triangulation $(K,\phi)$ of a closed and bounded definable set $S$ without \IC\ because it satisfies  $i)$ but not $ii)$.
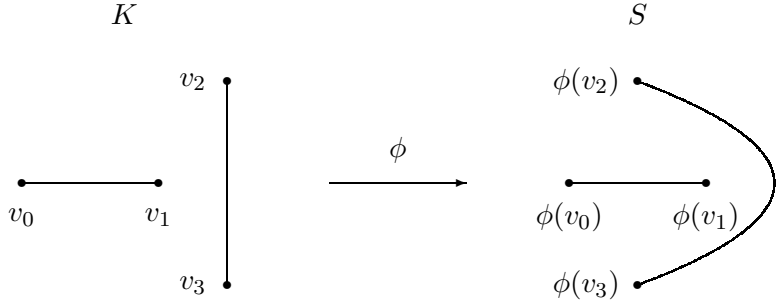
\begin{center}
\setlength{\unitlength}{0.9cm}
\begin{picture}(11,4.5)
\put(0,2){\line(1,0){2}}

\put(3,0.5){\line(0,1){3}}

\put(1.5,4.5){\makebox(0,0){$K$}}

\put(0,2){\circle*{0.1}} \put(2,2){\circle*{0.1}}
\put(3,0.5){\circle*{0.1}} \put(3,3.5){\circle*{0.1}}

\put(0,1.5){\makebox(0,0){$v_{0}$}}
\put(2,1.5){\makebox(0,0){$v_{1}$}}
\put(2.5,3.5){\makebox(0,0){$v_{2}$}}
\put(2.5,0.5){\makebox(0,0){$v_{3}$}}
\qbezier(9,3.5)(13,2)(9,0.5)

\put(8,2){\line(1,0){2}}

\put(9,4.5){\makebox(0,0){$S$}}

\put(8,2){\circle*{0.1}} \put(10,2){\circle*{0.1}}
\put(9,0.5){\circle*{0.1}} \put(9,3.5){\circle*{0.1}}

\put(8,1.5){\makebox(0,0){$\phi(v_{0})$}}
\put(10,1.5){\makebox(0,0){$\phi(v_{1})$}}
\put(8.25,3.5){\makebox(0,0){$\phi(v_{2})$}}
\put(8.25,0.5){\makebox(0,0){$\phi(v_{3})$}}
\put(5.5,2.5){\makebox(0,0){$\phi$}} \put(4.5,2){\vector(1,0){2}}
\end{picture}
\end{center}
3) The following example shows that we cannot deduce $iii)$ of \linebreak NT-triangulation from $i)$ and $ii)$.
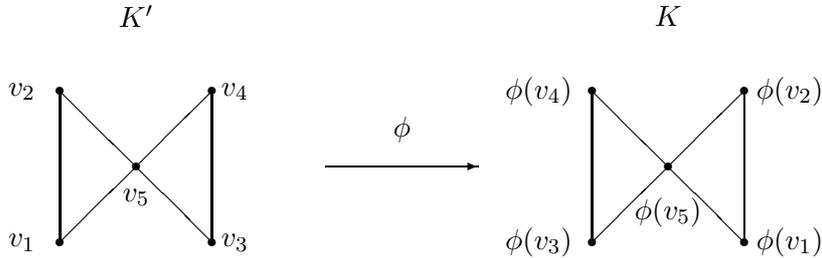
\begin{center}
\setlength{\unitlength}{1cm}
\begin{picture}(10,3.2)
\put(0.5,0){\line(0,1){2}}

\put(0.5,0){\line(1,1){1}}

\put(0.5,2){\line(1,-1){1}}

\put(1.5,1){\line(1,1){1}}

\put(1.5,1){\line(1,-1){1}}

\put(2.5,0){\line(0,1){2}}

\put(1.5,3){\makebox(0,0){$K'$}}

\put(0.5,0){\circle*{0.1}} \put(0.5,2){\circle*{0.1}}
\put(1.5,1){\circle*{0.1}} \put(2.5,0){\circle*{0.1}}
\put(2.5,2){\circle*{0.1}} 
\put(0,0){\makebox(0,0){$v_{1}$}}
\put(0,2){\makebox(0,0){$v_{2}$}}
\put(1.5,0.6){\makebox(0,0){$v_{5}$}}
\put(2.8,0){\makebox(0,0){$v_{3}$}}
\put(2.8,2){\makebox(0,0){$v_{4}$}}

\put(7.5,0){\line(0,1){2}}

\put(7.5,0){\line(1,1){1}}

\put(7.5,2){\line(1,-1){1}}

\put(8.5,1){\line(1,1){1}}

\put(8.5,1){\line(1,-1){1}}

\put(9.5,0){\line(0,1){2}}

\put(8.5,3){\makebox(0,0){$K$}}

\put(7.5,0){\circle*{0.1}} \put(7.5,2){\circle*{0.1}}
\put(8.5,1){\circle*{0.1}} \put(9.5,0){\circle*{0.1}}
\put(9.5,2){\circle*{0.1}} 

\put(6.8,0){\makebox(0,0){$\phi(v_{3})$}}
\put(6.8,2){\makebox(0,0){$\phi(v_{4})$}}
\put(8.5,0.4){\makebox(0,0){$\phi(v_{5})$}}
\put(10.1,0){\makebox(0,0){$\phi(v_{1})$}}
\put(10.1,2){\makebox(0,0){$\phi(v_{2})$}}

\put(5,1.5){\makebox(0,0){$\phi$}}

\put(4,1){\vector(1,0){2}}
\end{picture}
\end{center}
where $K'=K$ and $\phi$ is a symmetry.

\end{ejem}
\begin{obs}\label{obs:icprimerasub}The condition \VC\ is always easy to obtain taking the first barycentric subdivision. The problem is that we are interested in proving the existence of both properties (\IC\ and \VC) at the same time. In general it is not true that if $K$ has \IC\ then its first barycentric subdivision has \IC, as we can see in the following example,
\\
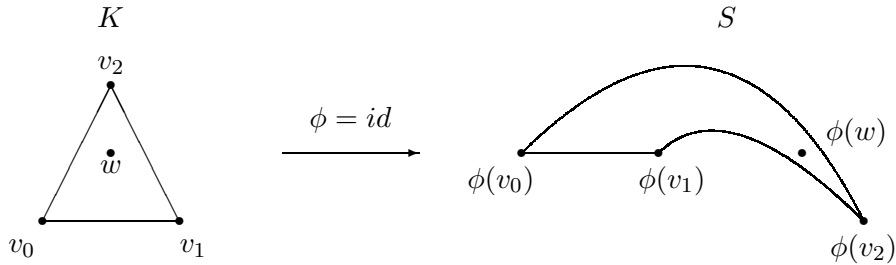
\begin{center}
\setlength{\unitlength}{0.9cm}
\begin{picture}(13,5)
\put(0.5,2){\line(1,2){1}}

\put(0.5,2){\line(1,0){2}}

\put(2.5,2){\line(-1,2){1}}

\put(1.5,5){\makebox(0,0){$K$}}

\put(0.5,2){\circle*{0.1}} \put(2.5,2){\circle*{0.1}}
\put(1.5,4){\circle*{0.1}} \put(1.5,3){\circle*{0.1}}

\put(1.5,2.8){\makebox(0,0){$w$}}
\put(0.2,1.6){\makebox(0,0){$v_{0}$}}
\put(2.7,1.6){\makebox(0,0){$v_{1}$}}
\put(1.5,4.3){\makebox(0,0){$v_{2}$}}

\put(7.5,3){\line(1,0){2}}

\qbezier(9.5,3)(10.5,4)(12.5,2)

\qbezier(7.5,3)(10.5,6)(12.5,2)

\put(10.5,5){\makebox(0,0){$S$}}

\put(7.5,3){\circle*{0.1}} \put(12.5,2){\circle*{0.1}}
\put(9.5,3){\circle*{0.1}} \put(11.6,3){\circle*{0.1}}

\put(12.4,3.3){\makebox(0,0){$\phi(w)$}}
\put(7.2,2.6){\makebox(0,0){$\phi(v_{0})$}}
\put(9.7,2.6){\makebox(0,0){$\phi(v_{1})$}}
\put(12.5,1.6){\makebox(0,0){$\phi(v_{2})$}}

\put(5,3.5){\makebox(0,0){$\phi=id$}}

\put(4,3){\vector(1,0){2}}
\end{picture}
\end{center}
\end{obs}

In section \ref{section:3} we will prove the existence of triangulations with \IC\ and \VC.

\begin{teoICVC}Let $S\subset R^{m}$ be a closed and bounded definable set and let
$S_{1},\ldots,S_{k}$ be definable subsets of $S$. Then there exists a triangulation $(K,\phi)\in \Delta(S;S_{1},\ldots,S_{k})$ with both \IC\ and \VC$(S_{1},\ldots,S_{k})$.
\end{teoICVC}

We will need the TISSP-triangulation theorem to prove the existence 
 of the NT-triangulations in section \ref{section:4}.
\begin{teoNT}\label{teo:triangulacion}Let $K$ be a closed simplicial complex and let $S_{1},\ldots,S_{k}$ be definable subsets of $|K|$. Then there exist a triangulation $(K',\phi')\in \Delta^{NT}(|K|;S_{1},\ldots,S_{k})$.
\end{teoNT}

We will use an induction argument to prove the existence of a triangulation with \IC. In the induction step we will need the existence of \VC. In this way we will obtain the TISSP-triangulation theorem in which triangulations with both properties (\IC\ and \VC) are proved to exists.
               %
               %
\section{\IC\ and \VC\ properties}\label{section:3}
In this section we will prove the TISSP-triangulation theorem. First we prove two lemmas.
\begin{lema}\label{lema:triangclausu}Let $S$ be a definable subset, $S_{1},\ldots,S_{k}$ definable subsets of $S$ and $(K,\phi)\in \Delta(S;S_{1},\ldots,S_{k})$. Then $(K,\phi)\in \Delta(S;cl_S(S_{i}))_{i=1,\ldots,k}$, so that $(K,\phi)\in \Delta(S;\partial_S S_{i})_{i=1,\ldots,k}$.
\end{lema}
\begin{proof}Let $\sigma_{1},\ldots,\sigma_{l}\in K$ be such that $S_{i}=\phi(\sigma_{1}) \dot{\cup} \cdots \dot{\cup} \phi(\sigma_{l})$. Then 
$$\begin{array}{lcl}
cl_S(S_{i}) & = & cl_S(\phi(\sigma_{1})) \cup \cdots \cup cl_S(\phi(\sigma_{l}))=\\
& = & \phi(cl_{|K|}(\sigma_{1}))\cup \cdots \cup \phi(cl_{|K|}(\sigma_{l}))=\\�
& = & \dot{\bigcup}_{\tau\in\mathcal{F}}\phi(\tau),
\end{array}$$
where $\mathcal{F}$ is the collection of simplices $\{\sigma_{1},\ldots,\sigma_{l}\}$ and all their faces.
\end{proof}
\begin{lema}\label{lema:VCcomobar}Let $S$ be a closed and bounded definable set and let $S_{1},\ldots,S_{k}$ be definable subsets of $S$. Let $(K,\phi)\in \Delta(S;S_{1},\ldots,S_{k})$ and let $(\tilde{K},\tilde{\phi})\in \Delta(S;\phi(\sigma))_{\sigma\in K}$ with \VC$(\partial\phi(\sigma):\sigma\in K)$. Then, 
\begin{itemize}
\item[a)]if $\tau \in \tilde{K}$ and $\sigma\in K$ are such that $\tilde{\phi}(\tau)\subset \phi(\sigma)$ then there exist a vertex $v\in Vert(\tau)$ with $\tilde{\phi}(v)\in \phi(\sigma)$, and

\item[b)]$(\tilde{K},\tilde{\phi})\in \Delta(S;S_{1},\ldots,S_{k})$ and satisfies \VC$(S_{1},\ldots,S_{k})$.
\end{itemize}
\end{lema}
\begin{proof}Firstly we observe that by Lemma \ref{lema:triangclausu} the triangulation $(\tilde{K},\tilde{\phi})$ is compatible with the subsets $\partial \phi(\sigma)$, $\sigma \in K$.
\begin{enumerate}
\item[a)]Let $\tau \in \tilde{K}$ and $\sigma \in K$. Suppose that $\tilde{\phi}(\tau)\subset \phi(\sigma)$. If $\sigma=\{v \}$, $v\in Vert(K)$, then it is obvious. Assume that $\sigma$ is not a vertex and that $\tilde{\phi}(v)\notin \phi(\sigma)$, for any $v\in Vert(\tau)$.
Hence, as $\overline{\partial \phi(\sigma)}=\partial \phi(\sigma)$, by \VC($\partial\phi(\sigma):\sigma\in K$), $\tilde{\phi}(\tau)\subset \partial \phi(\sigma)$, a contradiction.

\item[b)]That $(\tilde{K},\tilde{\phi})\in \Delta(S;S_{1},\ldots,S_{k})$ is clear. Let us show that $(\tilde{K},\tilde{\phi})$ has \VC($S_{1},\ldots,S_{k}$). So let $\tau=(v_{0},\dots,v_{n})\in \tilde{K}$ be such that $\tilde{\phi}(v_{i})\in \overline{S_j}$, $i=0,\ldots,n$. Let $\sigma\in K$ be such that $\tilde{\phi}(\tau)\subset \phi(\sigma)$. By $a)$ there exists a vertex $v_{i_{0}}\in Vert(\tau)$ with $\tilde{\phi}(v_{i_{0}})\in \phi(\sigma)$. By Lemma \ref{lema:triangclausu} $(K,\phi)$ is compatible with $\overline{S_{j}}$ and hence $\phi(\sigma)\subset \overline{S_{j}}$. Therefore $\tilde{\phi}(\tau)\subset \overline{S_{j}}$.
\end{enumerate}
\end{proof}
\begin{obs}Following the notation of Lemma \ref{lema:VCcomobar}, if $\tau\in \tilde{K}$ and $\sigma\in K$, $\sigma$ not a vertex, with $\tilde{\phi}(\tau)\subset \phi(\sigma)$ then $\tilde{\phi}(\tau)\neq \phi(\sigma)$.
\end{obs}

The proof of the TISSP-triangulation theorem is a refinement of the proof of the Triangulation theorem, 8.2.9 in \cite{MR1633348}. We will make extensively use of the following notions and results from Chapter 8 of \cite{MR1633348}. We have include them here since the notation is slightly different.

\begin{lema}\label{lema:usolevant}Let $(a_0,\ldots,a_n)$ be a $n$-simplex in $R^{m}$ and $r_j,s_j\in R$, $r_j\leqslant s_j$, for $j=0,\ldots,n$ and $r_j< s_j$ for some $j$. Consider $b_j=(a_j,r_j)$, $c_j=(a_j,s_j)\in R^{m+1}$. Then $(b_0,\ldots,b_j,c_j,\ldots,c_n)$ is a $(n+1)$-simplex provided $b_j \neq c_j$. In fact, the collection of $(n+1)$-simplices $(b_0,\ldots,b_j,c_j,\ldots,c_n)$, with $b_j\neq c_j$, and all their faces is a closed simplicial complex.
\end{lema}
\begin{proof}See Lemma 8.1.10 in \cite{MR1633348}.

\end{proof}
\begin{deff}Let $S$ be definable set and let $(K,\phi)\in \Delta(S)$. We define a \textbf{triangulated set} as the pair $(S,\phi(K))$ where $\phi(K)=\{\phi(\sigma):\sigma \in K\}$. Given a triangulated set $(S,\mathcal{P})$ and $C,D\in \mathcal{P}$ we call $D$ a \textbf{face} of $C$ if $D\subset cl(C)$, a \textbf{proper face} of $C$ if $D\subset cl(C)-C$ and a \textbf{vertex} of $C$ if it has dimension $0$.
\end{deff}
\begin{deff}\textbf{A multivalued function $F$ on the triangulated set $(S,\mathcal{P})$} is a finite collection of functions, $F=\{f_{C,i}:C\in \mathcal{P},1\leq i \leq k(C)\}$, $k(C)\geq 0$, each function $f_{C,i}:C \rightarrow R$ definable and $f_{C,1}<\ldots<f_{C,k(C)}$. We set
$$\begin{array}{l}
\Gamma(F)=\bigcup_{f\in F}\Gamma(f),\\
F|_{C}=\{f_{C,i}:1\le i \le k(C)\}, C\in \mathcal{P},\\
P^{F}=\{\Gamma(f):f\in F\}\cup
\{(f_{C,i},f_{C,i+1}):C\in \mathcal{P},1\le i < k(C)\},\\
S^{F}=\textrm{the union of the sets in } \mathcal{P}^{F}.\\
\end{array}$$

We call $F$ \textbf{closed} if for each pair $C,D\in \mathcal{P}$ with $D$ a proper face of $C$ and each $f\in F|_C$ there is $g\in F|D$ such that $g(y)=\lim_{x\to y} f(x)$ for all $y\in D$. Note that then each  $f\in F$, say $f\in F|C$, extends continuously to a definable function $cl(f):cl(C)\cap S \rightarrow R$ such that the restrictions of $cl(f)$ to the faces of $C$ in $\mathcal{P}$ belong to $F$.

We call $F$ \textbf{full} if is closed, $k(C)\geq 1$ for all $C\in \mathcal{P}$, and 
\begin{itemize}
 \item[1)]for each pair $C,D\in \mathcal{P}$ with $D$ a proper face of $C$ and each $g\in F|D$ we have $g=cl(f)|D$ for some $f\in F|C$, where $cl(f)$ is the continuous extension of $f$ to $cl(C)\cap S$,

 \item[2)]if $f_{1},f_{2}\in F|C$, $f_{1}\neq f_{2}$, then there exist at least one vertex of $C$ where $cl(f_{1})$ and $cl(f_{2})$ take different values.
\end{itemize}
\end{deff}
\begin{obs}The definition 8.2.5 in \cite{MR1633348} of a full multivalued function differs slightly from the one given here because there, only $1)$ is assumed to be satisfied. Let $F$ be a closed multivalued function on a triangulated set $(S,\phi(K))$ with $1)$ and $K'$ the first barycentric subdivision of $K$. Then the multivalued function $F'$ on the triangulated set $(S,\phi(K'))$, obtained by the restrictions of the functions in $F$ to the sets of $\phi(K')$ is full. The problem is that we cannot use this construction because we are interested in \IC\ and this property have a bad behavior with the first barycentric subdivision as we saw in Observation \ref{obs:icprimerasub}.
\end{obs}
\begin{lema}\label{lema:multivacerrada}Let $F$ be a multivalued function on the triangulated set $(S,\mathcal{P})$ such that $\Gamma(F)$ is closed in $S\times R$, and there is $M>0$ such that $\Gamma(F)\subset A\times [-M,M]$. Then $F$ is closed.
\end{lema}
\begin{proof}See Lemma 8.2.6 in \cite{MR1633348}.
\end{proof}
\begin{deff}Let $S$ be a definable set in $R^m$, $(K,\phi)\in \Delta(S)$ and $S'\subset S\times R$ a definable set. Then a triangulation $(L,\psi)\in \Delta(S')$ in $R^{n+1}$ is said to be a \textbf{lifting} of $(K,\phi)$ if $K=\{\pi_n(\sigma):\sigma\in K\}$ and the diagram 
\begin{displaymath}
\xymatrix{ |L|\ar[r]^\psi \ar[d]_{\pi_{n}} &
     S'\ar[d]^{\pi_{m}}\\
|K|  \ar[r]_{\phi} & S}
\end{displaymath}
commutes where $\pi_m$ and $\pi_n$ are the projections maps on the first $m$ and $n$ coordinates respectively.
\end{deff}
The proof of the Triangulation theorem carries an induction argument. In our case, that is, to prove TISSP-triangulation theorem, we will need the following lemma in the induction step. Its proof is an adaptation --taking care of TIP-- of that of Lemma 8.2.8 in \cite{MR1633348}.
\begin{lema}\label{lema:modlemadries}Let $A\subset R^{m+1}$ be a closed and bounded definable set and $(K,\phi)\in \Delta(A)$ a triangulation in $R^p$ with \IC. Let $F$ be a full multivalued function on $(A,\phi(K))$. Then $(K,\phi)$ can be lifted to a triangulation $(L,\psi)\in \Delta(A^F;C^l)_{C^l\in \phi(K)^F}$ in $R^{p+1}$ with \IC.
\end{lema}
\begin{proof}Fix a linear order in $Vert(K)$. We now construct $L$ and $\psi$ above each $C\in \phi(K)$. Let $C\in \phi(K)$ and let $a_0,\ldots,a_n$ the vertices if $\phi^{-1}(C)$ listed in the order we imposed on $Vert(K)$. Let $f<g$ be two successive members of $F|_C$. Put
\begin{displaymath}\left\{
\begin{array}{l}
r_{j}=cl(f)(\phi(a_{j})),\\
s_{j}=cl(g)(\phi(a_{j})),\\
b_{j}=(a_{j},r_{j}),\\
c_{j}=(a_{j},s_{j}).\\
\end{array}\right.
\end{displaymath}
As $F$ satisfies $ii)$ of fullness the hypotheses for Lemma \ref{lema:usolevant} hold. Let $L(f,g)$ be the complex in $R^{p+1}$ constructed in that lemma. Define the map $\psi_{f,g}^{-1}:[cl(f),cl(g)]\rightarrow |L(f,g)|$ by
$$\psi_{f,g}^{-1}(x,tcl(f)(x)+(1-t)cl(g)(x))=t\Phi_{b}(x)+(1-t)\Phi_{c}(x), 0\leq t\leq 1,$$
where $\Phi_{b}(x)$ and $\Phi_{c}(x)$ are the points of
$(b_{0},\ldots,b_{n})$ and $(c_{0},\ldots,c_{n})$ with the same affine coordinates
with respect to $b_{0},\ldots,b_{n}$ and
$c_{0},\ldots,c_{n}$ as $\phi^{-1}(x)$ has with respect to $a_{0},\ldots,a_{n}$. We now check that $\psi^{-1}_{f,g}$ is indeed continuous. The map $\psi^{-1}_{f,g}$ is a bijection
and it follows from Corollary 6.1.13 i) in \cite{MR1633348} that it is continuous.
By Corollary 6.1.12 in \cite{MR1633348}, $\psi^{-1}_{f,g}$ is a homeomorphism.
We also define for each $f\in F|C$ the complex $L(f)$ in $R^{p+1}$ as the $n$-simplex $(b_0,\ldots,b_n)$ with $b_j=(a_j,cl(f)(\phi(a_{j})))$, and all its faces. Then $\psi_f^{-1}:\Gamma(cl(f))\rightarrow |L(f)|$ is by definition the homeomorphism given by 
$$\psi_f^{-1}(x,cl(f)(x))=\Phi_b(x),$$
where $\Phi_b(x)$ is defined as before.

Let $L$ be the union of all complexes $L(f,g)$ and $L(f)$, $C\in \phi(K)$,
$f,g\in F|_{C}$ successive. Let $\psi:|L|\rightarrow A^{F}$ be the map
such that $\psi|_{L(f,g)}=\psi_{f,g}$ and $\psi|_{L(f)}=\psi_{f}$, for every $C\in
\phi(K)$ and $f,g\in F|_{C}$ successive. Then $L$ is a closed complex in $R^{p+1}$ and
the map $\psi$ is well-defined and bijective. It is easy to check that $\psi$ is indeed continuous. Therefore by Corollary 6.1.12 in \cite{MR1633348}, is a homeomorphism.
Finally let us check that the triangulation $(L,\psi)$  has \IC. It is enough to check it for every $(L(f,g),\psi_{f,g})$ and $(L(f),\psi_{f})$, with $f,g\in F|_C$ successive. Given $v\in Vert(L(f,g))$ then either $v=b_j$ or $v=c_j$ for some $j$. Hence either $\psi_{f,g}(v)=\psi_{f,g}(b_{j})=(\phi(a_{j}),r_{j}):=\tilde{b}_{j}$ or $\psi_{f,g}(v)=\psi_{f,g}(c_{j})=(\phi(a_{j}),s_{j}):=\tilde{c}_{j}$. Observe that by \IC\ of $(K,\phi)$ we have that $\phi(a_{0}),\ldots,\phi(a_{n})$ are affinely independent. Therefore by Lemma \ref{lema:usolevant} the $(n+1)$-simplices $(\tilde{b}_{0},\ldots,\tilde{b}_{j},\tilde{c}_{j},\ldots,\tilde{c}_{n})$ and all their faces is a closed simplicial complex. This is enough to prove that $(L(f,g),\psi_{f,g})$ has \IC. We can prove that $(L(f),\psi_{f})$ has \IC\ in a similar way.
\end{proof}
\begin{cor}\label{obs:lemamodvanden}Let $A\subset R^m$ be a closed and bounded definable set and $(\phi,K)\in \Delta(A)$ with \IC. Let $F$ be a full multivalued function on the triangulated set $(A,\phi(K))$. Let 
$(L,\psi)\in \Delta(A^F;C^l)_{C^l\in \phi(K)^{F}}$  in $R^{p+1}$ with \IC\ constructed in the proof of Lemma
\ref{lema:modlemadries}. Let $\tau\in L$ be such that
$\psi(\tau)\subset C^{l}$ for some $C^{l}\in \phi(K)^{F}$,
$C=\pi(C^{l})\in \phi(K)$, $\phi^{-1}(C)=(a_{0},\ldots,a_{n})$.
Then
\begin{enumerate}
\item[1)]$\pi(\psi(\tau))=C$,

\item[2)]if $C^{l}=(f,g)_{C}$, $f,g\in F|_{C}$
successive, then $\tau$ is either an $(n+1)$-simplex
$(b_{0},\ldots,b_{j},c_{j},\ldots,c_{n})$, $b_{j}\neq c_{j}$, or
is a $n$-simplex \linebreak $(b_{0},\ldots,b_{j-1},c_{j},\ldots,c_{n})$,
$b_{j}\neq c_{j}$, where $b_{j}=(a_{j},cl(f)(\phi(a_{j})))$ and
$c_{j}=(a_{j},cl(g)(\phi(a_{j})))$, $j=0,\ldots,n$,

\item[3)]if $C^{l}=\Gamma(f)$ then $\tau$ is
an $n$-simplex $(b_{0},\ldots,b_{n})$, where \linebreak
$b_{j}=(a_{j},cl(f)(\phi(a_{j})))$, $j=0,\ldots,n$. In fact,
$\psi(\tau)=\Gamma(f)=C^{l}$.
\end{enumerate}
\end{cor}
 
In the following lemma we show how we can achieve $2)$ of fullness of a multivalued function without losing \IC\ thanks to \VC.

\begin{lema}\label{lema:sustitucion}Let $A$ be a closed and bounded definable set, let $(K,\phi)\in \Delta(A)$ and $F$ a closed multivalued function on the triangulated set $(A,\phi(K))$ such that $1)$ of fullness is satisfied. Let $(K_{0},\phi_{0})\in \Delta(A;\phi(\sigma))_{\sigma\in K}$ with \VC($\partial \phi(\sigma):\sigma\in K$). Then, the multivalued function $F_{0}$ on $(A,\phi(K_0))$, obtained by the restrictions of the functions in $F$ to the sets of $\phi(K_0)$, is full.
\end{lema}
\begin{proof}Firstly we observe that by Lemma \ref{lema:triangclausu} the triangulation $(K_0,\phi_0)$ is compatible with the subsets $\partial \phi(\sigma)$, $\sigma \in K$. Clearly the multivalued function $F_{0}$ is closed and satisfies $1)$ of fullness. Let us check that $F_0$ satisfies also $2)$ of fullness. Let $C=\phi_{0}(\tau)$,
where $\tau\in K_{0}$, and let $f_{1},f_{2}\in F_{0}|_{C}$ be two different functions.  By construction, there exists 
$\sigma\in K$ and two different $\tilde{f}_{1},\tilde{f}_{2}\in F|_{\phi(\sigma)}$ such that $C\subset \phi(\sigma)$ and $\tilde{f}_{i}|_{C}=f_{i}$, $i=1,2$. By Lemma
\ref{lema:VCcomobar} a), there exists a vertex $\phi_{0}(v)$ of $C$ such that
$\phi_{0}(v)\in \phi(\sigma)$. Then $cl(f_{i})(\phi_{0}(v))=\tilde{f}_{i}(\phi_{0}(v))$, $i=1,2$, so 
$cl(f_{1})$ and $cl(f_{2})$ take different values on $\phi_{0}(v)$.
\end{proof}
\begin{teo}[TISSP-TRIANGULATION THEOREM]\label{tissptoerema}Let $S\subset R^{m}$ be a closed and bounded definable set and let
$S_{1},\ldots,S_{k}$ be definable subsets of $S$. Then there exists a triangulation $(K,\phi)\in \Delta(S;S_{1},\ldots,S_{k})$ with both \IC\ and \VC($S_{1},\ldots,S_{k}$).
\end{teo}
\begin{proof}By induction on $m$. The case $m=0$ is trivial. Suppose the theorem holds for a certain $m$ and let us prove it for $m+1$. Let $S\subset R^{m+1}$ be a closed and bounded definable set and let $S_1,\ldots,S_k$ be definable subsets of $S$. Consider the closed and bounded definable set 
$$T=bd(S) \cup
bd(S_{1}) \cup \cdots \cup bd(S_{k}).$$
So $dim(T)<m+1$ by Corollary 4.1.10 in \cite{MR1633348}. Therefore by Lemma 7.4.2 in \cite{MR1633348} and applying a certain linear automorphism we can assume that $e_{m+1}$ is a good direction for $T$. Observe that \IC\ and \VC\ are preserved by linear automorphisms.

Let $A=\pi(S)=\pi(T)$, which is a closed and bounded definable set in $R^{m}$. Since $e_{m+1}$ is a
good direction for $T$ and by Cell decomposition theorem, 3.2.11 in \cite{MR1633348} , the set
$T$ is the disjoint union of $\Gamma(f)$'s for finitely many definable functions $f$ on cells $A_h$ that form a partition of $A$. By inductive hypothesis there exists 
a triangulation $(K,\phi)$ compatible with the subsets $A_{h}$ with \IC\ and \VC($\pi(S_{i}):i=1,\ldots,k$).
The restrictions of the functions $f$ to the sets of $\phi(K)$ form a multivalued function $F$ on $(A,\phi(K))$ such that
$\Gamma(F)=T$. Since $T$ is closed and bounded, by Lemma
\ref{lema:multivacerrada} the multivalued function $F$ is closed.
However, $F$ may not be full. We achieve fullness with two modifications of the multivalued function $F$.
\begin{itemize}
\item[M1)]Since $F$ is closed each function $f\in F$ extends definably and continuously to the closure of its domain and then, by Lemma 8.2.2 in in \cite{MR1633348}, it extends definably and continuously to a function $\tilde{f}:A\rightarrow R$. Let $\tilde{T}$ be the union of the graphs $\Gamma(\tilde{f})$ for $f\in F$. By inductive hypothesis there exists a triangulation $$(K_1,\phi_1)\in \Delta(A;\phi(\sigma),\pi(S\cap \Gamma(\tilde{f})), \pi(S_i \cap \Gamma(\tilde{f})))_{\sigma\in K,f\in F,i=1,\ldots,k}$$ with \IC\ and \VC(${\partial \phi(\sigma)}:\sigma\in K$) . Observe that by Lemma \ref{lema:VCcomobar} b) the triangulation $(K_1,\phi_1)$ satisfies \VC($\pi(S_1),\ldots,\pi(S_k)$). Let $F_1$ be the multivalued function on $(A,\phi(K_1))$ obtained by the restrictions of the  extensions $\tilde{f}$, $f\in F$, to the the sets of $\phi(K_1)$. Since $\tilde{T}=\Gamma(F_1)$ and $\tilde{T}$ is closed and bounded then, by Lemma \ref{lema:multivacerrada}, $F_1$ is closed. As every function $f\in F$ has been extended, $F_1$ clearly satisfies $1)$  of fullness. Since $(K_1,\phi_1)$ is compatible with $\pi(S\cap \Gamma(\tilde{f})), \pi(S_i \cap \Gamma(\tilde{f}))$, $f\in F$, $i=1,\ldots,k$, then $\phi_1(K_1)^{F_1}$ is compatible with the sets $S_i$ and $\overline{S_i}$, $i=1,\ldots,k$. That is, $S_i$ and $\overline{S_i}$ are finite disjoint unions of sets of $\phi_1(K_1)^{F_1}$.

\item[M2)] By induction hypothesis there exists $(K_2,\phi_2)\in \Delta(A;\phi_1(\sigma))_{\sigma \in K_1}$ with \IC\ and \VC($\partial \phi_1(\sigma): \sigma\in K_1$). Let $F_2$ be the multivalued function on $(A,\phi(K_2))$ obtained by the restrictions of the functions in $F_1$ to the sets of $\phi_2(K_2)$. By Lemma \ref{lema:sustitucion} the multivalued function $F_2$ is full. Observe that by Lemma \ref{lema:VCcomobar} b) the triangulation $(K_2,\phi_2)$ satisfies \VC($\pi(S_1),\ldots,\pi(S_k)$).
\end{itemize}
Now we must make another modification to the multivalued function $F_2$ in order to achieve \IC\ and \VC($S_1,\ldots,S_k$).
\begin{itemize}
\item[M3)]Given $\tilde{C}\in \phi_2(K_2)$ and $\tilde{f},\tilde{g}\in F_2|_{\tilde{C}}$
successive with $\tilde{f}<\tilde{g}$, define the function
$\frac{\tilde{f}+\tilde{g}}{2}$ as
$\frac{\tilde{f}+\tilde{g}}{2}(x)=\frac{\tilde{f}(x)+\tilde{g}(x)}{2}$,
$x\in \tilde{C}$. Let $F_3$ the multivalued function obtained by adding to $F_2$ the new functions $\frac{\tilde{f}+\tilde{g}}{2}$ for each pair of successive functions $\tilde{f},\tilde{g}\in F_2|_{\tilde{C}}$,
$\tilde{C}\in \phi_2(K_2)$,
$\tilde{f}<\tilde{g}$.
\end{itemize}
The new multivalued function $F_3$ on $\phi_2(K_2)$ is also full. Let us show that $F_3$ satisfies the following property
\begin{displaymath}(*)\left\{
\begin{array}{l}
\textrm{if $\tilde{f}'\in F_3|_{\tilde{C}}$, $\tilde{C} \in \phi_2(K_2)$, is such that for every vertex $v$ of $\tilde{C}$}\\
\textrm{we have that $(v,cl(\tilde{f}')(v))\in \overline{S}_{i}$ then $\Gamma(\tilde{f}')\subset \overline{S}_{i}$,}\\
\end{array}\right.
\end{displaymath}
which we will need below to prove that the lifting of $(K_2,\phi_2)$ satisfies
\VC($S_1,\ldots,S_k$). Let $\tilde{f}'\in
F_3|_{\tilde{C}}$, $\tilde{C} \in
\phi_2(K_2)$, be such that
$(v,cl(\tilde{f}')(v))\in \overline{S}_{i}$ for every vertex $v$
of $\tilde{C}$. Suppose first that $\tilde{f}'\in
F_2|_{\tilde{C}}$ and $\Gamma(\tilde{f}')\nsubseteq
\overline{S_{i}}$. Since $F_2$ is the restrictions of the functions of $F_1$ to the sets of $\phi_1(K_1)$, there exists $f\in F_1|_{C}$, $\tilde{C}\subset C$, $C\in \phi_1(K_1)$, such that $f|_{\tilde{C}}=\tilde{f}'$. In fact, as
$\Gamma(\tilde{f}')\nsubseteq \overline{S_{i}}$, we have that
$\Gamma(f)\nsubseteq \overline{S_{i}}$. Hence
$\Gamma(f)\subset \overline{S_{i}}^{c}$. Since $(K_2,\phi_2)$ has \VC($\partial \phi_1(\sigma): \sigma\in K_1$), by Lemma
\ref{lema:VCcomobar} a) there exists one vertex $v$ of $\tilde{C}$ such that $v\in C$. Therefore
$(v,cl(\tilde{f}')(v))=(v,f(v))\in \Gamma(f)$ does not lie in
$\overline{S_{i}}$, which is a contradiction. Suppose now that
$\tilde{f}'=\frac{\tilde{f}+\tilde{g}}{2}$, for some successive functions $\tilde{f},\tilde{g}\in F_2|_{\tilde{C}}$.
Then $\Gamma(\tilde{f}')\subset
(\tilde{f},\tilde{g})_{\tilde{C}}$. Since $F_2$ is the restrictions of the functions of $F_1$ to the sets of $\phi_1(K_1)$, there exists $f,g\in
F_1|_{C}$, $\tilde{C}\subset C$, $C\in
\phi_1(K_1)$, such that $f|_{\tilde{C}}=\tilde{f}$ y
$g|_{\tilde{C}}=\tilde{g}$. Suppose
$\Gamma(\tilde{f}')\nsubseteq \overline{S_{i}}$. Then
$(\tilde{f},\tilde{g})_{\tilde{C}}\nsubseteq \overline{S_{i}}$ and therefore $(f,g)_{C}\nsubseteq \overline{S_{i}}$. Hence
$(f,g)_{C}\subset \overline{S_{i}}^{c}$. By Lemma
\ref{lema:VCcomobar} a), there exists one vertex $v$ of $\tilde{C}$ such that $v\in C$. Then
$(v,cl(\tilde{f}')(v))=(v,\frac{cl(\tilde{f})+cl(\tilde{g})}{2}(v))=(v,\frac{f+g}{2}(v))\in
(f,g)_{C}$ does not lie in $\overline{S_{i}}$, which is a contradiction.

By Lemma \ref{lema:modlemadries}, we can lift $(K_2,\phi_2)$ to $(L_{0} ,\psi_{0})\in \Delta(A^{F_3};C^l)_{C^l \in \phi_2(K_2)^{F_3}}$ with \IC. Finally, let $L=\{\sigma \in L_{0}:\psi_{0}(\sigma)
\subset S\}$. Then clearly $(L,\psi)\in \Delta(S;S_1,\ldots,S_k)$ with \IC, where $\psi=\psi_{0}|_{|L|}$.

We finish the proof by checking that $(L,\psi)$ has \VC($S_1,\ldots,S_k$). Let $\tau=(v_{0},\ldots,v_{n})\in L$ be such that $\psi(v_{r})\in
\overline{S_{i}}$ for all $r=0,\ldots,n$. By construction of the lifting in the proof of Lemma \ref{lema:modlemadries} we have that
$\psi(\tau)\subset C^{l}$, for some $C^{l}\in
\phi_2(K_2)^{F_3}$,
$C=\pi(C^{l})\in \phi_2(K_2)$. It follows from Corollary
\ref{obs:lemamodvanden}.$1)$ that
$\pi(\psi(\tau))=\pi(C^{l})=C\in \phi_2(K_2)$. First let us prove that $C\subset \pi(\overline{S_{i}})$. Because $\pi(\psi(v_{r}))\in \pi(\overline{S_{i}})=\overline{\pi(S_{i})}$
for all $r=0,\ldots,n$, $\pi(\psi(v_{r}))$  are the vertices of $C$ and $(K_2,\phi_2)$ satisfies \VC($\pi(S_{1}),\ldots,\pi(S_{k}))$, we have $C\subset
\overline{\pi(S_{i})}=\pi(\overline{S_{i}})$. 
We show now that
$\psi(\tau)\subset \overline{S_{i}}$. Consider the case that
$C^{l}$ is the graph of a function of $F_3$. Then, by Corollary \ref{obs:lemamodvanden}.$3)$, $C^{l}=\psi(\tau)$. It follows by $(*)$ that
$C^{l}\subset \overline{S_{i}}$. Now consider the case that $C^{l}=(f,g)_{C}$ for some functions $f,g\in
F_3|_{C}$ successive and suppose that $C^{l}\nsubseteq
\overline{S_{i}}$. Then $C^{l}\subset
\overline{S_{i}}^{c}$. By definition of $F_3$ we can assume that there exists $f_{1},g_{1}\in
F_2|_{C}$ successive such that $f=f_{1}$ and
$g=\frac{f_{1}|_{C}+g_{1}|_{C}}{2}$. Therefore $C^{l}\subset
D^{l}:=(f_{1},g_{1})_{C}$. Since $C^{l}\subset
\overline{S_{i}}^{c}$ then $D^{l}\subset
\overline{S_{i}}^{c}$.
\\

\emph{Claim. The set $$D^{l}_{cil}=\{(x,y):x\in
\overline{C}-C,cl(f_{1})(x)\neq cl(g_{1})(x), cl(f_{1})(x)<y<cl(g_{1})(x)
\}$$ is contained in $\overline{S_{i}}^{c}$.}
\\

Once we have proved the Claim, by Corollary
\ref{obs:lemamodvanden} $2)$, and following its notation, we have two cases:

\begin{enumerate}
\item[a)]$\tau$ is a $n$-simplex
$(b_{0},\ldots,b_{j-1},c_{j},\ldots,c_{n})$,$b_{j}\neq c_{j}$,
where $\phi_2^{-1}(C)=(a_{0},\ldots,a_{n})$,  or

\item[b)]$\tau$ is a $n$-simplex
$(b_{0},\ldots,b_{j},c_{j},\ldots,c_{n-1})$,$b_{j}\neq c_{j}$,
where $\phi_2^{-1}(C)=(a_{0},\ldots,a_{n-1})$.
\end{enumerate}

In any case, since $b_{j}\neq c_{j}$,
we have that $cl(f)(\phi_2(a_{j}))\neq
cl(g)(\phi_2(a_{j}))$ \linebreak and therefore
$cl(f_{1})(\phi_2(a_{j}))\neq
cl(g_{1})(\phi_2(a_{j}))$.  Hence it follows from \linebreak
$cl(g)(\phi_2(a_{j}))=(\frac{cl(f_{1})|_{C}+cl(g_{1})|_{C}}{2})(\phi_2(a_{j}))$ that
$$cl(f_{1})(\phi_2(a_{j}))<cl(g)(\phi_2(a_{j}))<cl(g_{1})(\phi_2(a_{j}))$$ and then $\psi(c_{j})=(\phi_2(a_{j}),cl(g)(\phi_2(a_{j})))\in D_{cil}^{l}$. This implies that $\psi(c_{j})\notin \overline{S_{i}}$, which is a contradiction. So then $\psi(\tau)\subset C^l \subset \overline{S_i}$ as required.

Finally we proof the claim.
\\
\hspace{-0.3cm}\emph{Proof of the Claim. }Suppose there exists $(x_{0},y_{0})\in D^{l}_{cil}$ such that $(x_{0},y_{0})\in
\overline{S_{i}}$. Since $\phi_2(K_2)^{F_2}$ is compatible with $\overline{S_i}$, for all $y\in
(cl(f_{1})(x_{0}),cl(g_{1})(x_{0}))$ we have that $(x_{0},y)\in
\overline{S_{i}}$. We show that if $y\in
(cl(f_{1})(x_{0}),cl(g_{1})(x_{0}))$ then $(x_{0},y)\in bd(S_{i})$.
Since $x_{0} \in \overline{C}- C$, by the Curve selection lemma, 6.1.5 in \cite{MR1633348}, there exists a curve $\gamma(t)$, $t\in (0,1)$, such that $\lim_{t \to 1}\gamma(t)=x_0$ and $\gamma(t)\in C$, $\forall t\in
(0,1)$. Consider the curve
$$\gamma_{y}(t)=(\gamma(t),f_{1}(\gamma(t))+\left( g_{1}(\gamma(t))-f_{1}(\gamma(t))\right) \left( \frac{y-f_{1}(x_{0})}{g_{1}(x_{0})-f_{1}(x_{0})} \right)), t\in (0,1).$$

Observe that $\gamma_{y}(t)\in D^{l}$, $t\in (0,1)$, and 
$\lim_{t \to 1}\gamma_y(t)=(x_0,y)$. Therefore $(x_{0},y) \in
\overline{D^{l}}\subset \overline{\overline{S_{i}}^{c}} \subset
\overline{int(S_{i}) ^{c}}=int({S_{i}})^{c}$. It follows that $(x_{0},y)\in bd(S_{i})$.

We have shown that $(x_{0},y)\in bd(S_{i})$ for all $y\in (cl(f_{1})(x_{0}),cl(g_{1})(x_{0}))$, which is a contradiction because $e_{m+1}$ is a good direction for $T$.
\end{proof}
\begin{obs}Following the notation of the proof of Theorem \ref{tissptoerema},
we show an example that explains the modification M3). Let $S$ be the following closed and bounded 2-dimensional definable set, where the union of the curves in its interior is the subset $S_1$.
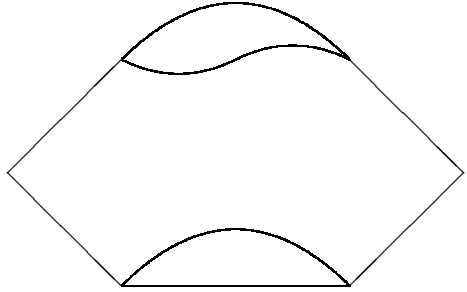
\setlength{\unitlength}{0.75cm}
\begin{center}
\begin{picture}(8,7)

\put(0.5,3){\line(1,1){2}} \put(0.5,3){\line(1,-1){2}}
\put(2.5,1){\line(1,0){4}} \put(6.5,5){\line(1,-1){2}}
\put(8.5,3){\line(-1,-1){2}} \qbezier(2.5,5)(4.5,7)(6.5,5)


\qbezier(2.5,5)(3.5,4.5)(4.5,5) \qbezier(4.5,5)(5.5,5.5)(6.5,5)
\qbezier(2.5,1)(4.5,3)(6.5,1)
\end{picture}
\end{center}
If we follow the proof of TISSP-triangulation theorem without the modification M3) then we obtain this triangulation
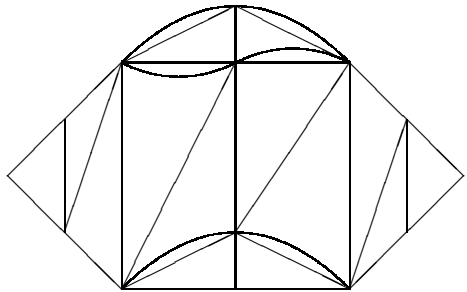
\begin{center}
\begin{picture}(8,7)
\qbezier(2.5,5)(4.5,7)(6.5,5) 
\put(0.5,3){\line(1,1){2}} \put(0.5,3){\line(1,-1){2}}
\put(2.5,1){\line(1,0){4}} \qbezier(2.5,5)(3.5,4.5)(4.5,5)
\qbezier(4.5,5)(5.5,5.5)(6.5,5) \put(6.5,5){\line(1,-1){2}}
\put(8.5,3){\line(-1,-1){2}} \qbezier(2.5,1)(4.5,3)(6.5,1)
\put(2.5,5){\line(1,0){4}} \put(2.5,5){\line(0,-1){4}}
\put(6.5,5){\line(0,-1){4}} \put(1.5,2){\line(0,1){2}}
\put(1.5,2){\line(1,3){1}} \put(7.5,2){\line(0,1){2}}
\put(6.5,1){\line(1,3){1}} \put(4.5,1){\line(0,1){5}}
\put(4.5,2){\line(2,3){2}} \put(2.5,1){\line(1,2){2}}
\put(2.5,5){\line(2,1){2}} \put(4.5,6){\line(2,-1){2}}
\put(2.5,1){\line(2,1){2}} \put(4.5,2){\line(2,-1){2}}
\qbezier(2.5,5)(3.5,4.5)(4.5,5)
\qbezier(4.5,5)(5.5,5.5)(6.5,5) \qbezier(2.5,1)(4.5,3)(6.5,1)
\end{picture}
\end{center}
which does not satisfy \VC($S_1$).
\end{obs}

            %
            %
\section{NT-triangulations}\label{section:4}

In this section we will prove the NT-triangulation theorem. The following are two well known results that we will use in the proof of the NT-triangulation theorem.
\begin{lema}\label{lema:complejo}Let $K$ be a finite collection of open simplices such that
\begin{enumerate}
\item[i)]every face of a simplex of $K$ is in $K$,

\item[ii)]every pair of distinct simplices of $K$ are disjoint.
\end{enumerate}
Then $K$ is closed simplicial complex.
\end{lema}
\begin{proof}It is essentially that of Lemma 1.2.1 in \cite{MR755006}, except that we work with open simplices. Let $\sigma$ and $\tau$ two different simplices of $K$.
We show that if $\overline{\sigma} \cap \overline{\tau}\neq
\emptyset$ then it equals the face $\sigma'$ of $\sigma$ that is spanned by those vertices $v_0,\ldots,v_m$ of $\sigma$ that are also vertices of $\tau$. As $\overline{\sigma}
\cap \overline{\tau}$ is convex and contains the vertices
$v_{0},\ldots,v_{m}$ we have that $\sigma' \subset
\overline{\sigma} \cap \overline{\tau}$. To prove the reverse inclusion, suppose $x\in
\overline{\sigma} \cap \overline{\tau}$. Then $x
\in s\cap t$, for some faces $s$ of $\sigma$ and $t$ of $\tau$. Because of $i)$ the simplices $s$ and $t$ are in $K$. As $s\cap t\neq \emptyset$ it follows from $ii)$
that $s=t$. Therefore the vertices of $s$ are also vertices of $\tau$, so that by definition, they are elements of the set $\{v_{0},\ldots,v_{m}\}$. Then $s$ is a face of $\sigma'$, so that $x\in \sigma'$.
\end{proof}
\begin{lema}\label{lema:contracsimpli}Let $\sigma$ be a $n$-simplex and $x\in \sigma$.
Then there exists a semialgebraic function
$$h:\overline{\sigma}\setminus \{x\}\rightarrow \partial \sigma$$
such that $h|_{\partial \sigma} =id$.
\end{lema}
\begin{proof}Let $y\in \partial \sigma$ and
$(x,y]=\{(1-t)x+ty:t\in (0,1]\}$. Define the semialgebraic function $h_{y}:(x,y]\rightarrow \partial \sigma $ such that
$h_{y}((1-t)x+ty)=y$, $t\in(0,1]$. It is enough to consider the semialgebraic function
 $h:\overline{\sigma}\rightarrow
\partial \sigma$ such that $h|_{(x,y]}=h_{y}$. It remains to check that $h$ is indeed continuous, which follows from Corollary 6.1.13 ii) in \cite{MR1633348}.
\end{proof}
\begin{teo}[NT-TRIANGULATION THEOREM]\label{teo:NTteorema}Let $K$ be a closed simplicial complex and let $S_{1},\ldots,S_{k}$ be definable subsets of $|K|$. Then there exist a triangulation $(K',\phi')\in \Delta^{NT}(|K|;S_{1},\ldots,S_{k})$.
\end{teo}
\begin{proof}By Theorem \ref{tissptoerema} there exist $(K_{0},\phi_{0})\in \Delta(|K|;S_1,\ldots,S_k,\sigma)_{\sigma\in K}$ satisfying \IC. Consider the collection of simplices $K'=\{\tau^{\phi_{0}}:\tau\in K_{0}\}$ (with the notation of Definition  \ref{def:IC}). Because of Lemma \ref{lema:complejo} and the fact that $K_0$ satisfy \IC,
we have that $K'$ is a closed simplicial complex. The map between the set of vertices
$$\begin{array}{rcl}
g_{vert}:Vert(K_{0}) & \rightarrow  & Vert(K')\\
v & \mapsto & \phi_0(v)\end{array}$$
induce a simplicial isomorphism $g:|K_{0}|\rightarrow |K'|$.
Observe that by definition given $\tau\in K_{0}$ we have that
$g(\tau)=\tau^{\phi_{0}}$. We also observe that given $\tau\in
K_{0}$ if $\phi_{0}(\tau)\subset \sigma \in K$
then the images of the vertices of $\tau$ by $\phi_{0}$ lie in $\overline{\sigma}$ and therefore $g(\tau)=\tau^{\phi_{0}}\subset \overline{\sigma}$. In
particular, given $\sigma\in K$ there exist
$\tau_{1},\ldots,\tau_{m}\in K_{0}$ such that
$\sigma=\phi_{0}(\tau_{1})\dot{\cup} \cdots \dot{\cup}
\phi_{0}(\tau_{m})$ and then 
\begin{eqnarray*}
g(\tau_{1})\dot{\cup} \cdots \dot{\cup}
g(\tau_{m})\subset\overline{\sigma}.
\end{eqnarray*}
\\
\emph{Claim: $g(\tau_{1})\dot{\cup} \cdots
\dot{\cup} g(\tau_{m})=\sigma.$}
\\

Once we have proved the Claim, we can assure that
\begin{enumerate}
\item[a)]$K'$ is a subdivision of $K$,

\item[b)]for every $\tau \in K_{0}$ and every $\sigma \in K$ we have that
$$\phi_{0}(\tau)\subset \sigma \Leftrightarrow g(\tau)\subset
\sigma.$$
\end{enumerate}
Finally we define the function
$$\phi':=\phi_{0}\circ g^{-1}:|K'| \rightarrow |K|.$$
Observe that $(K',\phi')\in \Delta(|K|;S_1,\ldots,S_k,\sigma)_{\sigma\in K}$ and that by $b)$ given a simplex $\tau^{\phi_{0}}\in K'$ if $\tau^{\phi_{0}}\subset \sigma \in K$ then
$\phi'(\tau^{\phi_{0}})=\phi_{0}\circ
g^{-1}(\tau^{\phi_{0}})=\phi_{0}(\tau)\subset \sigma$. Hence the three conditions of NT-triangulations are satisfied. It remains to prove the claim.
\\
\\
\emph{Proof of the Claim.} By induction on the dimension $n$ of the simplex $\sigma\in K$, the case $n=0$ being trivial.  Suppose the claim holds for a some $n$ and let us prove it for $n+1$. Let $\sigma\in K$ be a $(n+1)$-simplex and some
$\tau_{1},\ldots,\tau_{m}\in K_{0}$ such that $\sigma=\phi_{0}(\tau_{1})\dot{\cup} \cdots
\dot{\cup} \phi_{0}(\tau_{m})$. We first show that $g(\tau_{1})\dot{\cup} \cdots \dot{\cup} g(\tau_{m})\subset
\sigma$. Let $\tau_{m+1},\ldots,\tau_{l}\in K_0$ be such that $\partial \sigma =\phi_{0}(\tau_{m+1}) \dot{\cup} \cdots
\dot{\cup} \phi_{0}(\tau_{l})$. So $\overline{\sigma}=\phi_{0}(\tau_{1})\dot{\cup} \cdots \dot{\cup}
\phi_{0}(\tau_{m}) \dot{\cup} \phi_{0}(\tau_{m+1}) \dot{\cup}
\cdots \dot{\cup} \phi_{0}(\tau_{l})$. As $dim(\partial
\sigma)<n+1$, by induction hypothesis $\partial
\sigma=g(\tau_{m+1}) \dot{\cup} \cdots \dot{\cup} g(\tau_{l})$. Assume there exists $\tau_{j}$, $j\leq m$, such that
$g(\tau_{j})\cap \partial \sigma \neq \emptyset$. Then $g(\tau_{j})=g(\tau_{i})$, for some $i\geq m+1$.
In particular, as $g$ is a homeomorphism, we have that
$\tau_{i}=\tau_{j}$, which is a contradiction because
$\phi_{0}(\tau_{i})\cap \phi_{0}(\tau_{j})=\emptyset$.

We now show that $g(\tau_{1})\dot{\cup} \cdots
\dot{\cup} g(\tau_{m})=\sigma$. Suppose there exists $x\in
\sigma$ such that $x\notin g(\tau_{1})\dot{\cup} \cdots \dot{\cup}
g(\tau_{m})$. Following the notation of the previous paragraph, since $\partial \sigma=g(\tau_{m+1}) \dot{\cup} 
\cdots \dot{\cup} g(\tau_{l})$ we have that $(g\circ \phi_0^{-1})|_{\partial \sigma}:\partial \sigma \rightarrow \partial \sigma$ is a definable homeomorphism. Next, we consider the definable function $\varphi:= (g|_{\phi_0^{-1}(\overline \sigma)})\circ (\phi_{0}^{-1}|_{\overline \sigma})$,
$$\overline \sigma \overset{\phi_{0}^{-1}}{\longrightarrow} \phi_0^{-1}(\overline \sigma) \overset{g}{\longrightarrow} g(\phi_0^{-1}(\overline \sigma))\subset \overline{\sigma} \setminus \{x\}.$$
By Lemma \ref{lema:contracsimpli} there exists a semialgebraic function
$$h:\overline{\sigma} \setminus \{x\}\rightarrow \partial \sigma$$
such that $h|_{\partial \sigma}=id$. Therefore, as $g(\phi_0^{-1}(\overline \sigma))\subset \overline{\sigma} \setminus \{x\}$, the definable function
$$h \circ \varphi: \overline \sigma \rightarrow
\partial \sigma$$
is well-defined. Observe also that $h \circ \varphi \circ
i:\partial \sigma \rightarrow \partial \sigma$ is a definable homeomorphism, where $i:\partial \sigma \rightarrow \overline{\sigma}$ denotes the inclusion. Now we use the o-minimal homology theory (see section 3 in \cite{MR1883182}). Actually we are going to use the o-minimal reduced homology just to avoid considering separately the case $n=0$ (see section 1.7 of \cite{MR755006} for a discussion of the reduced homology theory). We will denote by $\tilde{H}_{*}(-)^{\mathcal{R}}$ the o-minimal reduced homology group. First observe that $\tilde{H}_{n}(\overline{\sigma})^{\mathcal{R}}=0$. Indeed, $\sigma$ is definably homeomorphic to a simplex whose vertices lies in $\mathbb{Q}$ so we can transfer the corresponding classical result (see Theorem 1.8.3 in \cite{MR755006}) by Proposition 3.2 in \cite{MR1883182}. Hence the induced homomorphism $\varphi_*:\tilde{H}_{n}(\overline{\sigma})^{\mathcal{R}}\rightarrow \tilde{H}_n(\overline{\sigma} \setminus \{x\})$ is identically zero. Since $\varphi_*=0$, we have that 
$$(h \circ \varphi \circ i)_{*}=h_{*} \circ \varphi_{*}
\circ i_{*}: \tilde{H}_{n}(\partial
\sigma)^{\mathcal{R}} \rightarrow \tilde{H}_{n}(\partial
\sigma)^{\mathcal{R}} $$
is also identically zero.

On the other hand, we observe that $\tilde{H}_{n}(\partial \sigma)^{\mathcal{R}}\neq 0$ because again we can transfer the corresponding classical result (see Theorem 1.8.3 in \cite{MR755006}) by Proposition 3.2 in \cite{MR1883182}. Since $h \circ \varphi \circ i:\partial \sigma \rightarrow \partial \sigma$ is a definable homeomorphism then the induced homomorphism $(h \circ \varphi \circ
i)_{*}:\tilde{H}_{n}(\partial
\sigma)^{\mathcal{R}} \rightarrow \tilde{H}_{n}(\partial
\sigma)^{\mathcal{R}}$ is an isomorphism, which is a contradiction.
\end{proof}
           %
                    %
\section{Properties of NT-Triangulations}\label{section:5}The next two propositions show the usefulness of condition $iii)$ in the definition of an NT-triangulation.

\begin{prop}\label{obs:tntriang}Let $K$ be a closed simplicial complex in $R^{m}$ and $S_{1},\ldots,S_{k}$ definable subsets of $|K|$. Let $(K',\phi')\in \Delta^{NT}(|K|;S_1,\ldots,S_k)$. Then given $\sigma\in K$ and $\tau_{1},\ldots,\tau_{l}\in K'$ with $\sigma=\phi'(\tau_{1})\dot{\cup} \cdots \dot{\cup}
\phi'(\tau_{l})$ we have that $\sigma=\tau_{1}\dot{\cup} \cdots
\dot{\cup} \tau_{l}$. In particular, for every subcomplex
 $L$ of $K$ we have that $\phi'(|L|)=|L|$.
\end{prop}
\begin{proof}Let $\sigma \in K$ and
$\tau_{1},\ldots,\tau_{l}\in K'$ with
$\sigma=\phi'(\tau_{1})\dot{\cup} \cdots \dot{\cup}
\phi'(\tau_{l})$. By $ii)$ and $iii)$ of the definition of NT-triangulations
we have that $\tau_{1}\dot{\cup} \cdots
\dot{\cup} \tau_{l}\subset \sigma$.
If $\tau_{1}\dot{\cup} \cdots
\dot{\cup} \tau_{l}\subsetneq \sigma$ then
by $ii)$ of NT there exists $\tau_{l+1}\in
K'$ with $\tau_{l+1}\subset \sigma$ and disjoint from $\tau_{i}$, $i=1,\ldots,l$.
Hence $\phi'(\tau_{l+1})\subset \sigma$ is disjoint from
$\phi'(\tau_{i})$, which is a contradiction. Therefore
$\sigma=\tau_{1}\dot{\cup} \cdots \dot{\cup} \tau_{l}$.
\end{proof}
\begin{prop}\label{lema:TNhomotopico}Let $K$ be a closed simplicial complex and $S_{1},\ldots,S_{k}$ definable subsets of $|K|$. Let $(K',\phi')\in \Delta^{NT}(S;S_1,\ldots,S_k)$. Then the definable homeomorphism $\phi':|K'|\rightarrow |K|$ is definably homotopic to the identity.
\end{prop}
\begin{proof}Consider the following definable map
$$\begin{array}{rcl}
H:|K'|\times I & \rightarrow &  |K'| \\
(x,s) & \rightarrow & (1-s)x+s\phi'(x).
\end{array}$$
The map $H$ is well-defined because, by $iii)$ of
NT of $(K',\phi')$, given $x\in \sigma \in K$ we have that
$\phi'(x)\in \sigma$. Observe also that $H$ is clearly continuous. Therefore $H$ is a definable homotopy between $\phi'$ and $id$ 
because
\begin{displaymath}\left\{
\begin{array}{l}
H(x,0)=x,  \ \ x\in |K'|,\\
H(x,1)=\phi'(x), \ \ x\in |K'|.\\
\end{array}\right.
\end{displaymath}
\end{proof}
\begin{cor}\label{cor:principal}Let $K$ be a closed simplicial complex and $S_1,\ldots,S_k$ definable subsets of $|K|$. Then there exists $(K',\phi')\in \Delta(|K|;S_1,\ldots,S_k,\sigma)_{\sigma \in K}$ such that $K'$ is a subdivision of $K$ and $\phi'$ is definably homotopic to $id$.
\end{cor}
\begin{proof}By Theorem \ref{teo:NTteorema} and Proposition \ref{lema:TNhomotopico}. 
\end{proof}

The next proposition is an extension property for NT-triangulations. Its proof is an adaptation of Lemma II.4.3 in \cite{MR819737} to the case of NT-triangulations.
\begin{prop}Let $K$ be a closed complex and $K_Z$ a closed subcomplex of $K$. Let $(K_0,\phi_0)\in \Delta^{NT}(|K_Z|)$. Then there exists $(K',\phi')\in \Delta^{NT}(|K|)$ such that $K_0\subset K'$ and $\phi'|_{|K_0|}=\phi_0$.
\end{prop}
\begin{proof}We first observe that it follows from Proposition \ref{obs:tntriang} that $|K_{0}|=|K_{Z}|$. For every $m\geq
0$ we denote by $SK^{m}$ the closed complex which is the union of $K_{Z}$ and all the simplices of $K$ of dimension $\leq m$. Let us show that there exists $(K^{m},\phi^{m})\in \Delta^{NT}(|SK^{m}|)$ such that
$K_{0} \subset K^{m}$ and
\begin{center}
$\phi^{m}|_{|K_{0}|}=\phi_{0}.$
\end{center}
Hence for $m=dim(K)$ we will obtain the NT-triangulation
$(K',\phi')$ as required.

For $m=0$ let $K^{0}$ be the union of $K_{0}$ and all vertices of $K$. Let $\phi^{0}$ be equal to $\phi_{0}$ on
$|K_{0}|$ and the identity on the vertices of $K$ that does not lie in $|K_{0}|$. Clearly $(K^{0},\phi^{0})\in \Delta^{NT}(|SK^0|)$, $K_{0}\subset K^{0}$ and $\phi^{0}|_{|K_{0}|}=\phi_{0}$.

Suppose we have constructed $(K^{m},\phi^{m})$. Let us construct $(K^{m+1},\phi^{m+1})$. Let $\Sigma_{m+1}$ be the collection of simplices
in $K-K_{0}$ of dimension $m+1$. For every
$\sigma\in \Sigma_{m+1}$ we have that $\partial \sigma
=\bar{\sigma}-\sigma$ is contained in $SK^{m}$. By $ii)$ of NT of $(K^{m},\phi^{m})$ there exists a finite collection of 
indices $J_{\sigma}$ and simplices $\tau^{\sigma}_j$ of
$K^{m}$, $j \in J_{\sigma}$, such that $\partial \sigma=
\dot{\bigcup}_{j\in J_{\sigma}}\tau^{\sigma}_j$. Consider the collection of simplices $\tau^{\sigma}_j$ of $K^{m}$ for each
$\sigma\in \Sigma_{m+1}$ and $j\in J_{\sigma}$. Define
for each $\sigma\in \Sigma_{m+1}$, $j\in J_{\sigma}$,
$$\begin{array}{crcl}
h^{\sigma}_j:& [\tau^{\sigma}_j,\hat{\sigma}]& \rightarrow & \overline{\sigma} \\
 & (1-t)u+t\hat{\sigma}& \rightarrow & (1-t)\phi^{m}(u)+t\hat{\sigma}
\end{array}$$
where $\hat{\sigma}$ denotes the barycenter of $\sigma$ and
$[\tau^{\sigma}_j,\hat{\sigma}]$ the cone over $\tau^{\sigma}_j$ with
vertex $\hat{\sigma}$, that is,
$$[\tau^{\sigma}_j,\hat{\sigma}]=\{(1-t)u+t\hat{\sigma}:u\in\tau_{j\sigma},t\in [0,1] \}.$$
Observe that $h^{\sigma}_j$ is well-define because given $u\in
\tau^{\sigma}_j$ there exists a proper face $\sigma_{0}\in K$ of $\sigma$ such that
$\tau^{\sigma}_j\subset \sigma_{0}$ and therefore, by $iii)$ of NT of $(K^{m},\phi^{m})$, we have that $\phi^{m}(u)\in \phi^{m}(\tau^{\sigma}_j)\subset \sigma_{0} \subset
\partial \sigma$. Hence $h^{\sigma}_j((1-t)u+t\hat{\sigma})\in \overline{\sigma}$ for all $t\in [0,1]$. We also observe that the map $h^{\sigma}_j$ is injective. We now check that $h^{\sigma}_j$ is indeed continuous, which follows from Corollary 6.1.13 ii) in \cite{MR1633348}. Finally let $K^{m+1}$ be the collection of simplices
$K^{m}$ and the collection of simplices
$$(\tau^{\sigma}_j,\hat{\sigma})=\{(1-t)u+t\hat{\sigma}:u\in\tau^{\sigma}_j,t\in (0,1)\}$$
and their faces. The triangulation $(K^{m+1},\phi^{m+1})$,
where $\phi^{m+1}|_{[\tau^{\sigma}_j,\hat{\sigma}]}=h^{\sigma}_j$,
is such that $K_{0}\subset K^{m+1}$ and
$\phi^{m+1}|_{|K_{0}|}=\phi_{0}$. We show now that
$(K^{m+1},\phi^{m+1})\in \Delta^{NT}(|SK^{m+1}|)$.
First observe that not only $\phi^{m+1}|_{|K_{0}|}=\phi_{0}$, but $\phi^{m+1}|_{|K_{m}|}=\phi^{m}$. Recall also that $(K^{m},\phi^{m})$ is NT. Let us check the three properties  of NT. To prove $i)$ of NT we observe that the image by $h^{\sigma}_j$ of every $(\tau^{\sigma}_j,\hat{\sigma})$ is
$$h^{\sigma}_j((\tau^{\sigma}_j,\hat{\sigma}))=(\phi^{m}(\tau^{\sigma}_j),\hat{\sigma}),$$
where
$(\phi^{m}(\tau^{\sigma}_j),\hat{\sigma})=\{(1-t)x+t\hat{\sigma}:x\in
\phi^{m}(\tau^{\sigma}_j), t\in (0,1)\}$. The sets
$\phi^{m}(\tau^{\sigma}_j)$ are pairwise disjoint and their union equals $\partial \sigma$. Therefore 
$$\tau^{\sigma}_i\neq \tau^{\sigma}_j\Rightarrow \tau^{\sigma}_i\cap \tau^{\sigma}_j=\emptyset \Rightarrow \phi^m(\tau^{\sigma}_i)\cap \phi^m( \tau^{\sigma}_j)=\emptyset \Rightarrow$$
$$\Rightarrow (\phi^{m}(\tau^{\sigma}_i),\hat{\sigma})\cap
(\phi^{m}(\tau^{\sigma}_j),\hat{\sigma})=\emptyset.$$
and the union of the
$(\phi^{m}(\tau^{\sigma}_j),\hat{\sigma})$ equals
$\overline{\sigma}$ without the point $\hat{\sigma}$. Condition $ii)$ of NT is clear because the cones
$[\tau^{\sigma}_j, \hat{\sigma}]$ and their faces form a triangulation
 of $\overline{\sigma}$. Clearly $iii)$ of NT holds since we have always worked inside each simplex $\sigma \in
\Sigma_{m+1}$.
\end{proof}

\begin{thebibliography}{99}

\bibitem[1]{miopreprint} El\'ias Baro, On o-minimal homotopy groups, preprint, 2007.
\bibitem[2]{MR1883182} Alessandro Berarducci and Margarita Otero, o-minimal fundamental group, homology and
manifolds, Journal of the London Mathematical Society (2), 65 (2002), 257-270.
\bibitem[3]{MR819737} Hans Delfs and Manfred Knebusch, \emph{Locally semialgebraic
spaces}, Lecture Notes in Mathematics, 1173, Springer-Verlag, Berlin, 1985.
\bibitem[4]{MR755006} James R. Munkres, \emph{Elements of algebraic topology}, Addison-Wesley, 1984.
\bibitem[5]{MR1633348} Lou van den Dries, \emph{Tame topology and o-minimal
structures}, London Mathematical Society Lecture Note Series, 248, Cambridge University Press, 1998.
\end{thebibliography}

E-mail address: elias.baro@uam.es.
\end{document}